\documentclass{emsprocart}
\usepackage{epsfig,overpic,hyperref}
\usepackage[all]{xy}
\usepackage{url}
\usepackage{mathrsfs}

\contact[t.coates@imperial.ac.uk]{Tom Coates, Department of
  Mathematics, Imperial College London, 180 Queen's Gate, London SW7
  2AZ, UK}

\contact[a.corti@imperial.ac.uk]{Alessio Corti, Department of
  Mathematics, Imperial College London, 180 Queen's Gate, London SW7
  2AZ, UK}

\contact[Sergey.Galkin@ipmu.jp]{Sergey Galkin, Kavli IPMU (WPI), The
  University of Tokyo, Kashiwa, Chiba 277--8583, Japan}

\contact[platforma.russia@gmail.com]{Vasily Golyshev, Algebra and
  Number Theory Sector, Institute for Information Transmission
  Problems, Bolshoy Karetny per.\ 19, Moscow 127994, Russia}

\contact[a.m.kasprzyk@imperial.ac.uk]{Alexander Kasprzyk, Department of
  Mathematics, Imperial College London, 180 Queen's Gate, London SW7
  2AZ, UK}

\newtheorem{thm}{Theorem}[section]

\newtheorem{ho}[thm]{Hope}

\theoremstyle{definition}
\newtheorem{dfn}[thm]{Definition}

\newtheorem{rem}[thm]{Remark}
\newtheorem{exa}[thm]{Example}

\DeclareMathOperator{\rk}{rk}
\DeclareMathOperator{\ev}{ev}
\DeclareMathOperator{\Identity}{I}

\DeclareMathOperator{\Vol}{Vol}

\DeclareMathOperator{\NE}{NE}
\DeclareMathOperator{\NA}{Amp}

\DeclareMathOperator{\Lie}{Lie}
\DeclareMathOperator{\Interior}{Int}
\DeclareMathOperator{\gkz}{gkz}
\DeclareMathOperator{\rf}{rf}
\DeclareMathOperator{\ord}{ord}

\DeclareMathOperator{\Lat}{Lattice}

\DeclareMathOperator{\Newt}{Newt}
\DeclareMathOperator{\Ker}{Ker}
\DeclareMathOperator{\Aut}{Aut}

\DeclareMathOperator{\Pic}{Pic}

\DeclareMathOperator{\Hom}{Hom}
\DeclareMathOperator{\sheafHom}{\underline{\mathit{Hom}}}

\DeclareMathOperator{\End}{End}
\DeclareMathOperator{\Spec}{Spec}
\DeclareMathOperator{\Sol}{Sol}

\newcommand{\oo}{\mathcal{O}}
\newcommand{\QQ}{\mathbb{Q}}
\newcommand{\RR}{\mathbb{R}}
\newcommand{\VV}{\mathbb{V}}
\newcommand{\PP}{\mathbb{P}}
\newcommand{\CC}{\mathbb{C}}
\newcommand{\TT}{\mathbb{T}}
\newcommand{\ZZ}{\mathbb{Z}}

\newcommand{\FF}{\mathbb{F}}

\newcommand{\Cstar}{\CC^\times}
\newcommand{\Cstarn}{(\CC^\times)^n}
\newcommand{\Cstarb}{(\CC^\times)^b}
\newcommand{\Cstarr}{(\CC^\times)^r}
\newcommand{\tti}{\mathtt{i}}

\title{Mirror Symmetry and Fano Manifolds}

\author[Tom Coates, Alessio Corti, Sergey Galkin, Vasily Golyshev,
Alexander Kasprzyk]{Tom Coates, Alessio Corti, Sergey Galkin, Vasily Golyshev,
Alexander Kasprzyk}

\begin{document}

\begin{classification}
Primary 14J33; Secondary 14J45, 14D07
\end{classification}

\begin{keywords}
  Fano manifolds, mirror symmetry, variations of Hodge structure,
  Landau--Ginzburg models.
\end{keywords}

\maketitle

\bibliographystyle{halpha}
\tableofcontents

\section{Introduction}
\label{sec:intro}

We give a sketch of mirror symmetry for Fano manifolds and we outline
a program to classify Fano \mbox{4-folds} using mirror symmetry. As
motivation, we describe how one can recover the classification of Fano
\mbox{3-folds} from the study of their mirrors. A glance at the table
of contents will give a good idea of the topics covered.  We take a
stripped-down view of mirror symmetry that originated in the work of
Golyshev~\cite{MR2306141} and that can also be found
in~\cite{MR2412270}.

\section{Local systems}
\label{sec:local-systems}

A \emph{local system} of rank $r$ on a (topological) manifold
$B$ is a locally constant sheaf $\VV$ of $r$-dimensional $\QQ$-vector
spaces. To give a local system is equivalent to give its monodromy
representation $\rho \colon \pi_1(B,x)\to \Aut \VV_x \cong GL_r (\QQ)$
where $x\in B$.  We write $r=\rk \VV$.  The central theme of this note
is the detailed comparison of two different ways that local systems
arise in mathematics.

All local systems in this note: (a) support---at least
conjecturally---an additional structure such as a (polarised)
variation of (pure) Hodge structure, or a structure of an $l$-adic
sheaf over a base $B$ defined over a number field. \footnote{It is natural to imagine
  that local systems with additional structures (realisations) subject
  to natural comparisons would be the object of a category of
  ``quantum'' motivic sheaves with a construction akin to
  \cite{MR1012168,MR1043451}, see \cite{MR2483750}. It is interesting
  to wonder what a Grothendieck-style definition of such a category
  might look like, and what it might mean.}; and (b) have an integral
structure, for instance they are local systems of free
$\ZZ$-modules. In particular we assume throughout that $\VV$ is
\emph{polarised}, i.e.~that it carries a nondegenerate symmetric or
antisymmetric bilinear form $\psi \colon \VV\otimes \VV \to \QQ$.

Let $C$ be a compact topological surface, $S\subset C$ a finite set,
and $\VV$ a local system on $U=C\setminus S$. We denote by $x\in U$ a
point and by $j\colon U=C\setminus S\hookrightarrow C$ the natural
(open) inclusion. If $s\in S$ and $\gamma_s \in \pi_1(U, x)$ is a loop
around $s$, then we write $T_s=\rho (\gamma_s)\in \Aut \VV_x$ for the
monodromy transformation; $T_s$ is defined only up to conjugation, but
this will be unimportant in what follows.

\begin{dfn}
  \label{dfn:1}
  The \emph{ramification} of $\VV$ is:
  \[
  \rf \VV = \sum_{s\in S} \dim (\VV_x/\VV_x^{T_s}) \,.
  \]
\end{dfn}

If $\VV$ as above is a local system on $U=C\setminus S$, and the genus
of $C$ is $g$, then, by Euler--Poincar\'e, $\rf \VV
+(2g-2)\rk\VV=-\chi(C, j_\star \VV)$.  If $\VV$ is nontrivial
irreducible, then $H^0(C, j_\star \VV)=\VV_x^{\pi_1(U, x)}=(0)$ and,
dually, also $H^2(C,j_\star \VV)=(0)$. Thus, if $C=\PP^1$ and $\VV$ is
nontrivial irreducible, then:
\[
\rf \VV - 2\rk \VV=-\chi(\PP^1;j_\star \VV)=h^1(\PP^1; j_\star
\VV)\geq 0\, .
\]
We call the quantity $\rf \VV - 2\rk \VV$ the \emph{ramification
  index} of $\VV$. Even from a purely topological perspective, local
systems with ramification index zero seem special. As far as we know,
to date there has been no systematic study of $l$-adic sheaves on
$\PP^1$ of ramification zero.

\section{Local systems from Laurent polynomials}
\label{sec:local-systems-from}

Local systems arise classically in algebraic geometry as the
cohomology groups of the fibers of a morphism $f\colon X \to B$.

\subsection*{The classical period of a Laurent polynomial}
\label{sec:classical-period}

We discuss the special case where $f\colon \Cstarn \to \CC$ is a
Laurent polynomial in $n$ variables, that is, an element of the
polynomial ring $\CC[x_1,x_1^{-1},\dots, x_n, x_n^{-1}]$ where
$x_1,\dots, x_n$ are the standard co-ordinates on $\Cstarn$.

\begin{dfn}
  \label{dfn:5}
  Let $f\colon \Cstarn \to \CC$ be a Laurent polynomial.  The
  \emph{classical period} of $f$ is:
  \[
  \pi_f(t)=\Bigl(\frac{1}{2\pi \tti}\Bigr)^n \int_{|x_1|=\cdots =|x_n|=1}
  \frac1{1-tf(x_1,\dots, x_n)}
  \frac{dx_1}{x_1}\cdots \frac{dx_n}{x_n}
  \]
\end{dfn}

\begin{thm}
  \label{thm:1}
  The classical period satisfies an ordinary differential equation
  $L\cdot \pi_f(t)\equiv 0$, where $L \in \CC\langle t, D\rangle$ is a
  polynomial differential operator and $D=t \frac{d}{dt}$.
\end{thm}

\begin{proof} 
  In short: our period $\pi_f (t)$ is a specialisation of integrals
  which are solutions of the differential systems introduced in
  \cite{MR1011353}, for which we recommend the
  survey~\cite{MR2306158}. We next explain this in greater detail. Let
  $P\subset \ZZ^n$ be the Newton polytope of $f$ and denote by
  $m_0,\dots, m_N\in P\cap \ZZ^n$ the lattice points in $P$.  If $P$
  does not contain the origin then the classical period is constant
  and there is nothing to prove, so we assume that $m_0 = 0$.  Write:
  \[
  f=\sum_{i=0}^N a_i x^{m_i}
  \]
  Reparametrizing $t$ if necessary, we reduce to the case where
  $a_0=0$.  Denote by $\iota\colon \ZZ^n \hookrightarrow \ZZ^{n+1}$
  the affine embedding ``at height $1$'': $\iota (m)=(1,m)$. Write
  $\mathbf{m}_i=\iota(m_i)$, $0 \leq i \leq N$, and let $A\colon
  \ZZ^{N+1}\to \ZZ^{n+1}$ be the homomorphism that maps the standard
  basis vector $\mathbf{e}_i$ to $\mathbf{m}_i$, $0 \leq i \leq
  N$. If:
  \[
  g=\sum_{i=0}^N u_i x^{m_i}
  \]
  is the generic Laurent polynomial with Newton polytope $P$, then it
  is well-known \cite{MR1269718, MR2306158} that the period:
  \[
  \Phi_g(u_0,\dots, u_N)=\Bigl(\frac{1}{2\pi\tti}\Bigr)^n \int \frac{1}{g}
  \frac{dx_1}{x_1}\cdots \frac{dx_n}{x_n}
  \]
  satisfies the holonomic differential system\footnote{That is, the system of
    differential equations:
    \[
    \begin{cases}
      0 = \prod_{l_i<0} \Bigl(\frac{\partial}{\partial
        u_i}\Bigr)^{-l_i}-\prod_{l_i>0} \Bigl(\frac{\partial}{\partial
        u_i}\Bigr)^{l_i} & \text{for $\mathbf{l}=
        (l_1,\dots,l_{N+1})\in \Ker A$}\\
        0 = {-\mathbf{c}} +\mathbf{m}_0\,u_0\frac{\partial}{\partial
          u_0}+\cdots+\mathbf{m}_N \,u_N\frac{\partial}{\partial u_N}
    \end{cases}
    \] More precisely the period satisfies the \emph{extended} GKZ
    system of \cite[\S3.3]{MR1316509} or, equivalently, the
    \emph{better behaved} GKZ system of \cite{BoHo}. In the important
    case when $P$ is a reflexive polytope, the standard GKZ is the
    same as the better behaved GKZ. The rank of the local system of
    solutions of the better behaved system is always the normalised
    volume $\Vol P$.}  $\gkz (A,\mathbf{c})$ where
  $\mathbf{c}=(-1,0,\dots, 0)$ \cite[\S2.5]{MR2306158}. To get the
  operator $L$, restrict the coefficients to $u_i=a_i$ for $i>0$,
  change the variable $u_0$ to $t=-1/u_0$, and note that $\pi_f(t) =
  u_0 \, \Phi_g(u_0,a_1,\ldots,a_n)$.
\end{proof}

\begin{dfn}
  \label{dfn:6}
  The Picard--Fuchs operator $L_f \in \CC\langle t, D\rangle$ is
  the operator:
  \begin{align*}
    L_f = \sum_{j=0}^{k} p_j(t) D^j
    && p_j \in \CC[t]
  \end{align*}
  such that $L_f\cdot \pi_f \equiv 0$, where $k$ is taken to be as
  small as possible and, once $k$ is fixed, we choose $L_f$ so that
  $\deg p_k$ is as small as possible.  This defines $L_f$ uniquely up
  to multiplication by a constant.  We say that the \emph{order} $\ord
  L_f$ of $L_f$ is $k$, and the \emph{degree} $\deg L_f$ is the
  maximum of $\deg p_0, \deg p_1, \ldots, \deg p_k$.
\end{dfn}

It is clear from what we said above that $\ord L_f \leq \Vol
P$. 

\begin{rem}
   \label{rem:2}
   The local system $\Sol L_f$ is an irreducible summand of the
   polarised variation of Hodge structure $\text{gr}^W_{n-1}R^{n-1}f_!
   \, \ZZ_{\Cstarn}$. By \cite[Thm~4.5]{Deligne}, $L_f$ has regular
   singularities.
 \end{rem}

\subsection*{How to compute the Picard--Fuchs operator and the ramification}
\label{sec:how-compute-picard}

Consider the \emph{period sequence} $(c_m)_{m \geq 0}$, where
$c_m=\text{coeff}_\mathbf{1}(f^m)$.  Expanding $\pi_f(t)$ as a power
series in $t$ and applying the residue theorem $n$ times yields:
\[
\pi_f(t)=\sum_{m=0}^\infty c_m t^m 
\]
Consider a polynomial differential operator $L=\sum t^k P_k(D)$ where
$P_k(D)\in \CC[D]$ is a polynomial in $D$; then $L\cdot \pi_f \equiv
0$ is equivalent to the linear \emph{recursion relation} $\sum
P_k(m-k)c_{m-k}=0$. In practice, to compute $L_f$ one uses knowledge
of the first few terms of the period sequence and linear algebra to guess
the recursion relation; note that the computation of $c_m$, say for
$1\leq m \leq 600$, is very expensive.  Given $L_f$, one can compute
$\rf (\Sol L_f)$ algorithmically using elementary Fuchsian
theory.

\begin{exa}
  \label{exa:1}
  If $f(x,y)=x+y+x^{-1} y^{-1}$, then:
  \[
  \pi_f(t)=\sum_{m\geq 0} \frac{(3m)!}{(m!)^3}t^{3m} 
  \]
  The coefficients satisfy the recursion relation:
  \[
  m^2 c_{3m}-3(3m-1)(3m-2)c_{3m-3}=0
  \]
  and, by what we said, this is equivalent to:
  \[
  \Bigl[D^2-27t^3(D+1)(D+2)\Bigr]\pi_f=0
  \]
  Studying this ODE, one finds that the ramification defect $\rf (\Sol
  L_f)-2\rk (\Sol L_f)$ is zero.
 \end{exa}

 \begin{exa}
   \label{exa:2}
   Consider $f(x,y)=x+ xy+y+x^{-1} y^{-1}$. In this case:
   \begin{multline*}
     L_f = 8D^2 -tD -t^2(5D+8)(11D+8) - 12t^3(30D^2+78D+47) \\
     -4t^4(D+1)(103D+147) -99t^5(D+1)(D+2)
   \end{multline*}
   and the ramification defect $\rf (\Sol L_f)-2\rk (\Sol L_f) =1$.
 \end{exa}

\section{Local systems from quantum cohomology}
\label{sec:local-systems-from-1}

Local systems also arise in the study of quantum cohomology, as
solutions of the regularised quantum differential equation.  When $X$
is a Fano manifold, the space of solutions of the regularised quantum
differential equation for $X$ defines a local system on
$\PP^1\setminus S$.

\subsection*{Fano manifolds}
Recall that a complex projective manifold $X$ of complex dimension $n$
is called \emph{Fano} if the anticanonical line bundle $-K_X=\wedge^n
T_X$ is ample.  If $n=2$, $X$ is called a \emph{del~Pezzo} surface. It
is well-known that a del~Pezzo surface is isomorphic to $\PP^1\times
\PP^1$ or the blow up of $\PP^2$ in $\leq 8$ general points: thus,
there are 10 deformation families of Fano manifolds in two dimensions.
There are 105 deformation families of 3-dimensional Fano manifolds: 17
families with $b_2=1$ and 88 families with $b_2\geq 2$
\cite{MR463151,MR503430,MR1022045,MR2112566}.  We state a theorem of
Mori that plays a crucial role in what follows:

\begin{thm}
  \label{thm:3}
  Let $X$ be a Fano manifold. Denote by $\NE X\subset H_2(X;\RR)$ the
  \emph{Mori cone} of $X$: that is, the convex cone generated by
  (classes of) algebraic curves $C\subset X$. Then $\NE X$ is a
  rational polyhedral cone.
\end{thm}

\subsection*{The quantum period of a Fano manifold}
\label{sec:quantum-period-1}

When $X$ is Fano, denote by $X_{0,k,m}$ the moduli space of stable
morphisms $f\colon (C,x_1,\dots,x_k) \to X$ where $C$ is a curve of
genus $0$ with $k$ marked points $x_1,\dots,x_k\in C$, and $\deg
f^\star(-K_X)=m$. This moduli space has \emph{virtual dimension}
$m-3+n+k$.  Here we are mainly interested in $X_{0,1,m}$ and the
\emph{evaluation morphism} at the marked point:
\[
\ev \colon X_{0,1,m}\to X 
\] 
Denote by $\psi$ the first Chern class of the \emph{universal
  cotangent line bundle} on $X_{0,1,m}$, that is, the relative
dualising sheaf $\omega_\pi$ of the forgetful morphism $\pi \colon
X_{0,1,m} \to X_{0,0,m}$.

\begin{dfn}
  \label{dfn:3}
  The \emph{quantum period} of $X$ is the power series $G_X
  (t)=\sum_{m \geq 0} p_m t^m$ where $p_0 = 1$, $p_1 = 0$, and
  $p_m=\int_{X_{0,1,m}}\psi^{m-2}\ev^\star (\text{pt})$ for $m \geq
  2$.  The sequence $(p_m)_{m \geq 0}$ is the \emph{quantum period
    sequence}.
\end{dfn}

\begin{thm}
  \label{thm:2}
  The quantum period of a Fano manifold $X$ satisfies a ordinary
  differential equation $Q \cdot G_X(t)\equiv 0$, where $Q \in \ZZ
  \langle t, D\rangle$ is a polynomial differential operator and $D=t
  \frac{d}{dt}$.
\end{thm}

\begin{proof}
  In short: our quantum period $G_X(t)$ is a specialisation of one
  component of the \emph{small $J$-function}. The result then follows
  from general properties of quantum cohomology going back to
  Dijkgraaf. We next recall the relevant facts from the theory of
  quantum cohomology\footnote{See \cite{MR2391365,MR1702284} for more
    comprehensive treatments.} and explain this in greater detail.

  In what follows we denote by $X_{0,k,\beta}$ the moduli space of
  stable morphisms of class $\beta \in \NE X \cap H_2(X,\ZZ)$.  Recall
  that the \emph{small quantum product} $a\ast b$ of even degree
  cohomology classes $a,b\in H^\bullet (X;\CC)$ is defined by the
  following formula, which is to hold for all $c\in H^\bullet(X;\CC)$:
  \[
  (a\ast b, c)=
  \sum_{\beta \in \NE X \cap H_2(X;\ZZ)}q^\beta
  \langle a,b,c\rangle_{0,3,\beta}
  \]
  where $(a,b)=\int_X a \cup b$ is the Poincar\'e pairing, $q^\beta$
  lies in the group ring $\CC[H_2(X;\ZZ)]$\footnote{In general we
    should work with the subgroup $H_2(X)^{\text{alg}}\subset H_2(X)$; here and
    in the rest of the paragraph we use the fact that if $X$ is Fano
    manifold then $H_2(X)=H_2^{\text{alg}}(X)$.}, and:
  \[
  \langle a, b, c\rangle_{0,3,\beta}=\int_{X_{0,3,\beta}}
  \ev_1^\star (a) \cup \ev_2^\star (b) \cup \ev_3^\star (c)
  \]
  is the 3-point correlator.  The structure of the small quantum
  product is equivalent to an \emph{integrable algebraic connection}
  $\nabla$ on:
  \begin{itemize}
  \item the trivial bundle with fiber the even part
    $H^\text{ev}(X;\CC)$ of $H^\bullet(X;\CC)$, over
  \item the torus $\TT=\Spec \CC[H_2(X,\ZZ)]$.
  \end{itemize}
  In other words $\TT$ is the torus with character group
  $\Hom_\text{groups}(\TT, \Cstar)=H_2(X;\ZZ)$, co-character group
  $\Hom_\text{groups}(\Cstar, \TT)=H^2(X;\ZZ)$, and group of
  $\CC$-valued points $\TT(\CC)=\Cstar \otimes H^2(X;\ZZ)$. Note
  that $\Lie \TT=H^2(X;\CC)$. The connection $\nabla$ is defined by:
  \begin{align*}
    &\nabla_X s = X\cdot s - X\ast s
    &&
    \text{where $s \colon \TT \to H^{\text{ev}}(X;\CC)$ and $X\in \Lie \TT= H^2(X;\CC)$.}
  \end{align*}
  The fact that this connection is \emph{algebraic} globally on $\TT$
  (in fact, the coefficients of the connection are polynomials)
  follows from the fact that quantum cohomology is \emph{graded} and
  that $-K_X>0$ on $\NE X$.  The fact that the connection is
  integrable (flat) is a fundamental property of quantum cohomology:
  it follows from the WDVV equations. Integrability means that the
  action of $\Lie \TT$ on $M=\{s \colon \TT \to H^\text{ev}(X;\CC)\}$
  extends to an action of the ring $D$ of differential operators on
  $\TT$: in other words $M$ is a $D$-module, called the \emph{quantum
    $D$-module}. $M$ therefore defines a $\mathcal{D}$-module
  $\mathcal{M}$, that is, a sheaf of modules $\mathcal{M}$ over
  the sheaf of differential operators $\mathcal{D}$ on $\TT$. In
  general, given a $\mathcal{D}$-module $\mathcal{M}$, one can form the local
  system $\sheafHom_{\mathcal{D}}(\mathcal{M},\oo)$ of
  \emph{solutions}\footnote{Here $\oo$ and $\mathcal{M}$ are sheaves of
    $\mathcal{D}$-modules in the analytic topology on $\TT$, and
    $\sheafHom$ is the sheaf of homomorphisms.}  of $\mathcal{M}$. Sections of
  this local system tautologically satisfy algebraic PDEs.

  Recall that the small $J$-function of $X$ is:
  \[
  J_X(q)= 1 + \sum_{\substack{\beta \in \NE X \cap H_2(X;\ZZ) \\ \beta
      \ne 0}} q^\beta J_\beta \in H^{\ev}(X;\CC)
  \]
  where $J_\beta =\ev_\star \big(\frac1{1-\psi}\big)$,
  $\ev \colon X_{0,1,\beta} \to X$ is the evaluation map at the marked point,
  and we expand $\frac1{1-\psi}$ as a power series in $\psi$.  It is
  well-known that $J_X(q)$ is a solution of the quantum
  $\mathcal{D}$-module and therefore it tautologically satisfies an
  algebraic PDE. Note that $J_X(q)$ is cohomology-valued but it makes
  sense to take its \emph{degree-zero} component $J_X^0(q)\in
  H^0(X;\CC)$; we can regard $J_X^0(q)$ as a $\CC$-valued
  function, because $H^0(X;\CC)$ is canonically generated by the identity
  class $\mathbf{1}$.

  Finally, the anticanonical class ${-K_X}\in H^2(X;\ZZ)$ is a
  co-character of $\TT$, that is, ${-K_X}$ gives a group homomorphism
  which we denote $\kappa\colon \Cstar \to \TT$. Since
  $G_X(t)=J_X^0\circ \kappa (t)$, where $t$ is the co-ordinate
  function on $\Cstar$, the discussion above makes it clear that
  $G_X(t)$ satisfies an algebraic ODE.
\end{proof}

\begin{dfn}
  \label{dfn:8} 
  The \emph{quantum differential operator} of $X$ is the operator
  $Q_X\in \ZZ\langle t, D\rangle$ of lowest order, as in
  Definition~\ref{dfn:6}, such that $Q_X \cdot G_X(t) \equiv 0$.
\end{dfn}

\subsection*{How to compute $Q_X$}
\label{sec:how-calculate-q_x}

In practice one starts by fixing a basis $\{T^a\}$ of
$H^\text{ev}(X;\ZZ)$ with $T^0=\mathbf{1}$ the identity class. Let
$M=M(t)$ be the matrix of quantum multiplication by $-K_X$ in this
basis, written as a function on $\Cstar$ by composing with $\kappa
\colon \Cstar \to \TT$. Next consider the differential equation on
$\Cstar$:
\[
\begin{cases}
D \Psi (t) = \Psi (t)M(t) \\
\Psi (0) = \Identity  
\end{cases}
\]
for $\Psi \colon \Cstar \to \End \bigl(H^\text{ev}(X;\CC) \bigr)$
a matrix. (Note: tautologically, the differential $\kappa_\star \colon
\Lie \Cstar \to \Lie \TT$ sends $D= t \, \frac{d}{dt}$ to $-K_X \in 
H^2(X; \CC)=\Lie \TT$.) Then the first
\emph{column} of $\Psi$ is $J_X\circ \kappa (t)$; the first entry of
the first column is our quantum period $G_X(t)$.

\begin{exa}
  \label{exa:3}
  Consider $X=\PP^2$ with cohomology ring $\CC[P]/P^3$, where $P$ is
  the first Chern class of $\mathcal{O}(1)$. Choose the basis $\mathbf{1},
  -K_X=3P, K_X^2=9 \, \text{pt}$ for the cohomology. The matrix of
  quantum multiplication by $-K_X$, in this basis, is:
  \[
  M=
  \begin{pmatrix}
    0 & 0 & 27 t^3 \\
    1 & 0 & 0 \\
    0 & 1 & 0
  \end{pmatrix}
  \]
  where the coefficient of $t^3$ in the upper right corner of the matrix is
  calculated as a nontrivial Gromov--Witten number:
  \[
  \langle -K_X\ast (K_X^2), \text{pt} \rangle_{0,3,[\text{line}]} =3
  \langle K_X^2,\text{pt} \rangle_{0,2,[\text{line}]}= 
  27 \langle \text{pt},\text{pt} \rangle_{0,2,[\text{line}]} = 27
  \]
  Next we consider the system:
  \[
  D(s_0, s_1, s_2)=(s_0,s_1,s_2) M
  \]
  The column $s_0$ is annihilated by the differential operator 
  $Q_X= D^3-27t^3$, and so $G_X(t)=\sum_{m=0}^\infty
  \frac{t^{3m}}{(m!)^3}$.
\end{exa}

\subsection*{Computing $G_X$ using the quantum Lefschetz theorem}
\label{sec:calc-toric-manif}

We explain how to calculate the quantum period of a Fano complete
intersection in a toric manifold using the quantum Lefschetz theorem
of Kim, Lee, and Coates--Givental.  For us, a toric variety is a GIT
quotient $X=\CC^r/\!\!/_\chi \Cstarb$ where $\Cstarb$ acts via the
composition of a group homomorphism $\rho\colon \Cstarb \to \Cstarr$
with the canonical action of $\Cstarr$ on $\CC^r$. The group
homomorphism $\rho$ is given dually by a $b\times r$ integral matrix:
\[
D=(D_1,\dots D_r) \colon \ZZ^r \to \ZZ^b
\]
that we call the \emph{weight data} of the toric variety $X$.  The
weight data alone do not determine $X$: it is necessary to choose a
\emph{stability condition}, i.e.~a $\Cstarb$-linearized line bundle
$L$ on $\CC^r$.  This choice is equivalent to the choice of a
character $\chi\in \ZZ^b$ of $\Cstarb$; denoting by $L_\chi$ the
corresponding line bundle, we have:
\[
H^0(\CC^r;L_\chi)^{\Cstarb}
=\Bigl\{ f\in \CC[x_1,\dots,x_r] : \text{$f(\lambda
\,x)=\chi(\lambda)\,f(x)$ for all $\lambda \in \Cstarb$} \Bigr\}
\]
Having made this choice, the set of stable points is:
\[
U^s(\chi) = \Bigl\{\mathbf{a} \in \CC^r : 
\text{$\exists N\gg 0$ and $\exists f \in 
H^0(\CC^r;L_\chi^{\otimes N})^{\Cstarb}$ such that $f(\mathbf{a}) \neq 0$}\Bigr\}
\]
The set of $\chi\in \ZZ^b$ for which $U^s(\chi)$ is non-empty
generates a rational polyhedral cone in $\RR^b$ equipped with a
partition into locally closed rational polyhedral chambers defined by
requiring that $U^s(\chi)$ depends only on the chamber containing
$\chi$. We always choose $\chi$ in the interior of a chamber of
maximal dimension, and then define $X=U^s(\chi)/\Cstarb$. Under the
identification $\ZZ^b=H^2(X;\ZZ)=\Pic(X)$ the chamber containing
$\chi$ is identified with the \emph{ample cone} $\NA X$; in this way
too we regard the columns $D_i$ of the weight data $D$ as elements of
$H^2(X)$. The appropriate Euler sequence shows that $-K_X=\sum_{i=1}^r
D_i$.

\begin{thm}\cite{MR1653024}
  \label{thm:4}
  Let $X$ be a toric Fano manifold. Then
  \[
  G_X(t)=\sum_{\mathbf{k}\in \ZZ^b\cap \NE X} t^{-K_X\cdot
    \mathbf{k}}\frac1{(D_1\cdot \mathbf{k})!\cdots (D_r\cdot
    \mathbf{k})!} \,.
  \]\qed
\end{thm}

\begin{thm}
  \label{thm:5}
  Let $\FF$ be a Fano toric manifold and let $L_1, \dots, L_c$ be nef
  line bundles on $\FF$ such that $A=-(K_{\FF}+\sum_{i=1}^c L_i)\in
  \NA \FF$.  Let $X$ be a smooth complete intersection of codimension
  $c$ in $X$, defined by the equation $f_1=\cdots =f_c=0$ where
  $f_i\in H^0(\FF;L_i)$.  Let:
  \[
  F_X(t)=\sum_{\mathbf{k}\in \ZZ^b\cap \NE \FF} t^{A\cdot
    \mathbf{k}}\frac{(L_1\cdot \mathbf{k})!\cdots (L_c\cdot
    \mathbf{k})!}{(D_1\cdot \mathbf{k})!\cdots (D_r\cdot \mathbf{k})!}
  \]
  and let $a_1$ be such that $F_X = 1 + a_1 t + O(t^2)$.  Then
  $G_X(t)=\exp (-a_1t)F_X(t)$.
\end{thm}

\begin{proof}
  Combine Theorem~\ref{thm:4} with \cite{MR2276766}.
\end{proof}

\subsection*{The regularised quantum period and mirror symmetry}
\label{sec:regul-quant-peri}

  The operator $Q_X$ has a pole of order 2 (an irregular
  singularity) at $\infty$, thus it cannot directly be compared with
  $L_f$. This suggests the following definitions:

\begin{dfn}
  \label{dfn:8}
  The \emph{regularised quantum period} is the Fourier--Laplace
  transform $\widehat{G}_X(t)=\sum (m!)p_mt^m$ of the quantum period
  $G_X(t)$.  The \emph{regularised quantum differential operator} of
  $X$ is the operator $\widehat{Q}_X\in \ZZ\langle t, D\rangle$ of
  lowest order, as in Definition~\ref{dfn:6}, such that
  $\widehat{Q}_X \cdot \widehat{G}_X(t) \equiv 0$.
\end{dfn}

\begin{dfn}
  \label{dfn:4}
  The Laurent polynomial $f$ is \emph{mirror-dual} to the Fano
  manifold $X$ if $\pi_f(t)=\widehat{G}_X(t)$ or, equivalently, if
  $L_f=\widehat{Q}_X$.
\end{dfn}

With this definition a Fano manifold has infinitely many mirrors if it
has any at all. The relationship between different mirrors of
del~Pezzo surfaces is investigated in \cite{galkin10:_mutat,
  galkin12}, where it is shown that the different mirror Laurent
polynomials $f$ are related by cluster transformations, and together
define a global function on a cluster variety.

\section{Extremal local systems and extremal Laurent polynomials}
\label{sec:extr-local-syst}

Which local systems arise from the quantum cohomology of Fano
manifolds?  Golyshev first made the observation that there are
effective bounds on the ramification of the regularised quantum local
system $\VV=\Sol \widehat{Q}_X$ of a Fano manifold $X$.

\begin{dfn} \cite{Gol}
  \label{dfn:2}
  A local system $\VV$ on $C=\PP^1\setminus S$ is \emph{extremal} if
  it is irreducible, nontrivial, and $\rf \VV = 2 \rk \VV$.  A Laurent
  polynomial $f$ is extremal if the local system $\Sol L_f$ of
  \emph{solutions} of the ODE $L_f \cdot ()\equiv 0$ is extremal.  We
  write ELP for ``extremal Laurent polynomial''.
\end{dfn}

The regularised quantum local system of any 3-dimensional Fano
manifold is extremal. We believe that extremal motivic sheaves and
Laurent polynomials are interesting in their own right. It would be
nice to work out a topological classification of integral polarised
extremal local systems.

\begin{exa}
  \label{exa:3}
  Consider a semistable rational elliptic surface $f\colon X\to \CC$.
  In general $f$ has $12$ singular fibers. Beauville~\cite{MR664643}
  classified surfaces with the smallest possible number, $4$, of
  singular fibers. On each of these $X$, it is easy to find an open
  set $(\Cstar)^2 \cong U \subset X$ such that $f|_{U}$ is an extremal
  Laurent polynomial.
\end{exa}

\section{Examples in low dimensions}
\label{sec:mink-polyn-hodge}

We describe two classes of Laurent polynomials: Minkowski polynomials
(MPs) and Hodge--Tate polynomials. (For simplicity we describe these
only when the number of variables involved is $2$ or $3$.)\phantom{.}
MPs are especially nice because: (a) they are, experimentally and
conjecturally, of low ramification; and (b) any 3-dimensional Fano
manifold with very ample tangent bundle is mirror-dual to a MP.

\subsection*{The Minkowski ansatz}
\label{sec:mink-polyn}

Let $P$ be a lattice polytope.  Then $P\cap \ZZ^n$ generates an affine
lattice whose underlying lattice we denote by $\Lat (P)$.

\begin{dfn}
  \label{dfn:10} A lattice polytope $P$ is \emph{admissible} if the
  relative interior of $P$ contains no lattice points.  A lattice
  polytope $P\subset \RR^n$ is \emph{reflexive} if the following two
  conditions hold:
  \begin{itemize}
  \item[(a)] $\Interior P \cap \ZZ^n=\{\mathbf{0}\}$;
  \item[(b)] the polar polytope:
    \[
    P^\ast = \left\{f\in (\RR^n)^\ast : \text{$\langle f, v\rangle \geq
        -1$ for all  $v \in P$}\right\}
    \] 
    is a lattice polytope.
  \end{itemize}
\end{dfn}

\begin{dfn}
  \label{dfn:11}
  Let $Q\subset \RR^n$ be a lattice polytope. A \emph{lattice
    Minkowski decomposition} of $Q$ is a decomposition of $Q$ as a
  Minkowski sum $Q=R+S$ of lattice polytopes $R$, $S$ such that $\Lat
  (Q)=\Lat (R) + \Lat (S)$.
\end{dfn}

Fix a reflexive polytope $P\subset \RR^n$ of dimension $\leq 3$.  We
describe a recipe, the \emph{Minkowski ansatz}, to write down Laurent
polynomials:
\[
f=\sum_{\mathbf{m}\in P \cap \ZZ^n} a_\mathbf{m} \, x^\mathbf{m}
\]
with $\Newt (f)=P$. We need to explain how to choose the coefficients
$a_\mathbf{m}$. In all cases we take $a_{\mathbf{0}}=0$; this is a
normalisation choice that corresponds to the fact that $p_1=0$.  If
$F\subset P$ is a face of $P$, the \emph{face term} corresponding to
$F$ is the Laurent polynomial:
\[
f_F=\sum_{\mathbf{m}\in F \cap \ZZ^n} a_\mathbf{m} \, x^\mathbf{m} 
\]
If $P$ is a reflexive polygon then we just need to specify the edge
terms. If $E=[\mu, \mu+e\nu]$ is an edge of $P$, where $\nu$ is
primitive, we take the corresponding term to be
$f_E=x^\mu(1+x^\nu)^e$.  If $P$ is a reflexive 3-tope, then we treat
the edges as just said.  It remains to specify the face terms
$f_F$. First, lattice Minkowski decompose each face into irreducibles:
\[
F=F_1+\cdots +F_r
\]
We say that such a decomposition is admissible if all $F_i$ are
admissible. Assuming that each face of $P$ has an admissible
decomposition, fix such a decomposition: then we take the face term to
be: $f_F=\prod f_{F_i}$ where $f_{F_i}$ is given by putting
coefficients on the edges of $F_i$ exactly as above.  Note that the
Minkowksi ansatz can associate to a reflexive 3-tope $P$ more than one
Laurent polynomial (if one or more faces of $P$ admit more than one
admissible decomposition), or exactly one Laurent polynomial (if every
face of $P$ admits a unique admissible decomposition), or no Laurent
polynomial (if some face of $P$ admits no admissible decomposition).

\subsection*{MP in 2 variables}
\label{sec:mps-2-variables}

There are 16 reflexive polygons and each supports exactly one
MP. This gives 16 MPs but only 8 distinct (classical) period
sequences. These are the quantum period sequences of the del~Pezzo
surfaces of degree $\geq 3$, that is, of the del~Pezzo surfaces with very
ample anti-canonical bundle. The 8 period sequences are extremal with two
exceptions: the first we already met in Example~\ref{exa:2} (the
mirror of $\FF_1$), and the other is:

\begin{exa}[the mirror of $\text{dP}_7$]
  \label{exa:4}
  $f(x,y)=x+y+x^{-1}+y^{-1}+x^{-1}y^{-1}$. Here: 
  \begin{multline*}
    L_f= 7D^2+tD(31D-3)-t^2(85D^2+238D+112)-2t^3(358D^2+785D+425)\\
    -2t^4(D+1)(669D+970)-731t^5(D+1)(D+2) 
  \end{multline*}
and the ramification defect $\rf (\Sol L_f) - 2\rk (\Sol L_f)$ is
equal to $1$.
\end{exa}

\subsection*{MP in 3 variables} In 3 variables, we have
(\url{http://www.fanosearch.net}):
\begin{itemize}
\item there are 4,319 reflexive 3-topes \cite{MR1663339};
\item they have 344 distinct facets, and these have 79 lattice
  Minkowski irreducible pieces;
\item of these, the admissible ones are $A_n$-triangles for $1\leq
  n\leq 8$;
\item MPs supported on reflexive 3-topes give rise to only 165
  (classical) period sequences. They are all extremal.
\end{itemize}

\begin{exa}
  \label{exa:5}
  Consider the reflexive polytope in $\RR^3$ with vertices given by
  the columns of:
  \[
  \left(\begin{array}{rrrrrr}
      1 & 0 & 0 & -2 & -3 & -1 \\
      0 & 1 & 0 & 0 & -1 & -1 \\
      0 & 0 & 1 & -1 & -1 & 1
    \end{array}\right)
  \]
  (This is the polytope with id 519664 in the GRDB database of toric
  canonical Fano 3-folds \cite{GRDB}.)\phantom{.}  The pentagonal
  facet has two Minkowski decompositions, and hence the polytope
  supports two Minkowski polynomials:
  \begin{align*}
    & f_1 = x + y + z + 3 x^{-1} + x^{-1}y^{-1}z + x^{-2}z^{-1} + 2
    x^{-2}y^{-1} + x^{-3}y^{-1}z^{-1}  \\
    & f_2 = x + y + z + 2 x^{-1} + x^{-1}y^{-1}z + x^{-2}z^{-1} + 2
    x^{-2}y^{-1} + x^{-3}y^{-1}z^{-1} 
  \end{align*}
  The classical periods associated to $f_1$ and $f_2$ begin as:
  \begin{align*}
    & \pi_1(t) = 1 + 6t^{2} + 90t^{4} + 1860t^{6} + 44730t^{8} +
    1172556t^{10} + \cdots \\
    & \pi_2(t) = 1 + 4t^{2} + 60t^{4} + 1120t^{6} + 24220t^{8} +
    567504t^{10} + \cdots
  \end{align*}
  and the corresponding Picard--Fuchs operators are:
  \begin{multline*}
    L_1 = 144 t^{4} D^{3} + 864 t^{4} D^{2} + 1584 t^{4} D - 40 t^{2}
    D^{3} + 864 t^{4} \\- 120 t^{2} D^{2} - 128 t^{2} D + D^{3} - 48
    t^{2}
  \end{multline*}
  \begin{multline*}
    L_2 = 128 t^{4} D^{3} + 768 t^{4} D^{2} + 1408 t^{4} D + 28 t^{2}
    D^{3} + 768 t^{4} + 84 t^{2} D^{2} \\+ 88 t^{2} D - D^{3} + 32
    t^{2} 
\end{multline*}  
\end{exa}

\subsection*{Hodge--Tate polynomials}
\label{sec:not-all-elp}

Let $f$ be a Laurent polynomial in 3 variables with Newton polytope
$P$, let $F$ be a facet of $P$, and let $f_F$ be the corresponding
face term.  Let $X_F$ be the toric variety corresponding to the
polygon $F$.  The equation $f_F=0$ defines a curve in $X_F$.  If $f$
is a MP then each such curve is of genus zero, thus MPs are
Hodge--Tate in the following sense.

\begin{dfn}
  \label{dfn:7}
  A 3-variable Laurent polynomial $f$ with Newton polytope $P$ is
  \emph{Hodge--Tate} if for all facets $F\subset P$, the curve
  $f_F=0$ has geometric genus zero.
\end{dfn}

One might hope that Hodge--Tate polynomials are of low ramification. 

\begin{exa}
  \label{exa:6}
  Consider the pictured polygon.
\begin{figure}[hpbt]
  \centering
  \begin{picture}(80,80)(-40,-40)
    \multiput(-30,-30)(30,0){3}{\makebox(0,0){$\cdot$}}
    \multiput(-30,0)(30,0){3}{\makebox(0,0){$\cdot$}}
    \multiput(-30,30)(30,0){3}{\makebox(0,0){$\cdot$}}
    \put(0,-30){\makebox(0,0){$\bullet$}}
    \put(0,0){\makebox(0,0){$\bullet$}}
    \put(-30,30){\makebox(0,0){$\bullet$}}
    \put(0,30){\makebox(0,0){$\bullet$}}
    \put(30,30){\makebox(0,0){$\bullet$}}
    \put(0,-30){\line(1,2){30}}
    \put(0,-30){\line(-1,2){30}}
    \put(-30,30){\line(1,0){60}}
    \put(42,6){\makebox(0,0){$\scriptstyle (1,0)$}}
    \put(12,38){\makebox(0,0){$\scriptstyle (0,1)$}}
  \end{picture}
\end{figure}
This is one of the smallest faces for which the Minkowski ansatz has
nothing to say. Consider the Laurent polynomial with this Newton
polygon given by $f=y(x^{-1}+2+x)+y^{-1}+a$.  For generic $a$ (the
completion of) $f=0$ is a curve of geometric genus $1$; it becomes of
genus $0$ exactly when $a \in \{-4,0,4\}$. Let us take $a=4$ and use
this as a new ``puzzle piece'' for assembling a Laurent polynomial.

Consider the 3-dimensional reflexive polytope with id 547363 in the
GRDB database of toric canonical Fano 3-folds \cite{GRDB}. This
polytope has
four faces: two smooth triangles, one $A_2$-triangle, and one face
isomorphic to the polygon shown above.  The corresponding Laurent
polynomial is:
\[
F = x+y+z+x^{-4} y^{-2} z^{-1}+2 x^{-2} y^{-1}+4 x^{-1}
\]
It has period sequence:
\[
1, 0, 8, 0, 120, 0, 2240, 0, 47320, 0, \ldots
\]
and Picard--Fuchs operator:
\[
512 t^{4} D^{3} + 3072 t^{4} D^{2} + 5632 t^{4} D - 48 t^{2} D^{3} +
3072 t^{4} - 144 t^{2} D^{2} - 160 t^{2} D + D^{3} - 64 t^{2} 
\]
The Laurent polynomial $F$ is Hodge--Tate but is not a MP.  It is
extremal, and is of manifold type in the sense of
\S\ref{sec:mink-polyn-fano}, but is not mirror-dual to any
3-dimensional Fano manifold.
\end{exa}

\section{Minkowski polynomials and Fano \mbox{3-folds}}
\label{sec:mink-polyn-fano}

Recall that in 3 variables there are 165 Minkowski (classical) period
sequences and, correspondingly, 165 Picard--Fuchs operators. We write
$L_f=\sum t^k P_k(D)$ where $P_k(D)\in \CC[D]$ is a polynomial in $D$,
and denote by $L_f(0)=P_0(D)$ the operator at $t=0$.  It turns out
that, if $L_f$ is one of the 165 Minkowski Picard--Fuchs operators,
then $L_f(0)$ splits as a product of linear factors over the
rationals. We say that $L_f$ is of \emph{manifold type} if all the
roots are integers; otherwise we say that $L_f$ is of \emph{orbifold
  type}.  Exactly 98 of the Minkowski Picard--Fuchs operators are of
manifold type and we have verified, by direct computation of
invariants on both sides, that they mirror the 98 deformation families
of 3-dimensional Fano manifolds $X$ such that $-K_X$ is very ample.
It will be interesting to see if the Minkowski Picard--Fuchs operators
of orbifold type mirror Fano \emph{orbifolds}.

It is natural to ask what invariants of a Fano manifold $X$ can be
computed from the knowledge of the differential operator
$\widehat{Q}_X$ alone.  This is a subtle question \cite{EHX, vEvS},
but in the case of \mbox{3-folds} we have good numerical evidence for
the following:

\begin{ho}[Galkin, Golyshev, Iritani, van Straten] Let $X$ be a
  3-dimensional Fano manifold and let $J_X(t)$ and $J_X^0(t)$ be as
  defined above (in the proof of Theorem~\ref{thm:2}). Then:
  \[
  \lim_{t\to +\infty}
  \frac{J_X(t)}{J_X^0(t)}=\widehat{\Gamma}(T_X)
  \]  
  where the limit is taken as $t$ tends to $+\infty$ along the real
  axis.  The characteristic class $\widehat{\Gamma}(T_X)$ is defined
  in \cite{MR2483750,MR2553377}.
\end{ho}

We briefly mention a promising line of thought. Consider a
\mbox{3-fold} toric Gorenstein canonical singularity $X_\sigma$, so
that $\sigma =\RR_+ (\iota \, F)$ where $F\subset \ZZ^2$ is a a
lattice polygon and $\iota \colon \ZZ^2 \to \ZZ^3$ is an affine
embedding at height one as above.  According to \cite{MR1452429},
deformation components of the singularity correspond to Minkowski
decompositions of $F$.  This suggests that Minkowski polynomials $f$
with $\Newt f =P$ may correspond to \emph{smoothing components} of the
singular toric Fano \mbox{3-fold} $X$ with fan polytope $P$. It would
be nice to make this precise, and to interpret the Minkowski
polynomials in terms of holomorphic disk counts in the framework of
Hori, Gross--Siebert, or Kontsevich--Soibelman.

\section{Fano 4-folds?}
\label{sec:fano-4-folds}

In 4 dimensions, there are over $473$ million reflexive
polytopes. Building on the Kreuzer--Skarke classification
\cite{MR1894855}, we are now in the process of making a database of
facets and of computing their lattice Minkowski decompositions.  We
plan to classify: Minkowski polynomials (and more general low
ramification Laurent polynomials) in 4 variables; their period
sequences; and their Picard--Fuchs operators.  This will give a list
of candidate families of Fano \mbox{4-folds}, and we aim to: compute
the (conjectural) invariants of these Fano \mbox{4-folds} assuming
that they exist; and construct the Fano explicitly in many
cases. Eventually, we hope turn this story into a classification
theory.

\subsection*{Acknowledgements}

This research is supported by TC's Royal Society University Research
Fellowship; ERC Starting Investigator Grant number~240123; the
Leverhulme Trust; Kavli Institute for the Physics and Mathematics of
the Universe (WPI); World Premier International Research Center
Initiative (WPI Initiative), MEXT, Japan; Grant-in-Aid for Scientific
Research (10554503) from Japan Society for Promotion of Science and
Grant of Leading Scientific Schools (N.Sh. 4713.2010.1); and EPSRC
grant EP/I008128/1 ``Extremal Laurent Polynomials''.

\pagestyle{plain}
\def\cprime{$'$}

\end{document}